\documentclass[12pt]{amsart}
\usepackage{amssymb}
\begin{document}
\numberwithin{equation}{section}

\def\1#1{\overline{#1}}
\def\2#1{\widetilde{#1}}
\def\3#1{\widehat{#1}}
\def\4#1{\mathbb{#1}}
\def\5#1{\frak{#1}}
\def\6#1{{\mathcal{#1}}}
\def\8#1{\MakeLowercase{#1}}
\def\9#1{\MakeUppercase{#1}}

\def\C{{\4C}}
\def\R{{\4R}}
\def\N{{\4N}}
\def\Z{{\4Z}}

\title[Jet vanishing orders]
{
	Jet vanishing orders and effectivity 
	of Kohn's algorithm in dimension $3$\\
	\medskip
	\sl
	\8{
		\9Dedicated to \9Professor \9Ngaiming \9Mok
		on the occasion of his $60$th birthday
	}
}


\author[S.-Y. Kim \& D. Zaitsev]{Sung-Yeon Kim* and Dmitri Zaitsev**}
\address{S.-Y. Kim: Center for Mathematical Challenges, Korea Institute for Advanced Study, 85 Hoegiro, Dongdaemun-gu,
Seoul, Korea }
\email{sykim8787@kias.re.kr}
\address{D. Zaitsev: School of Mathematics, Trinity College Dublin, Dublin 2, Ireland}
\email{zaitsev@maths.tcd.ie}
\thanks{*This research was supported by Basic Science Research Program through the National Research Foundation of Korea (NRF) funded by the Ministry of  Science, ICT and Future Planning (grant number NRF-2015R1A2A2A11001367)}

\def\Label#1{\label{#1}{\bf (#1)}~}
\def\Label#1{\label{#1}}


\def\cn{{\C^n}}
\def\cnn{{\C^{n'}}}
\def\ocn{\2{\C^n}}
\def\ocnn{\2{\C^{n'}}}


\def\dist{{\rm dist}}
\def\const{{\rm const}}
\def\rk{{\rm rank\,}}
\def\id{{\sf id}}
\def\aut{{\sf aut}}
\def\Aut{{\sf Aut}}
\def\CR{{\rm CR}}
\def\GL{{\sf GL}}
\def\Re{{\sf Re}\,}
\def\Im{{\sf Im}\,}
\def\span{\text{\rm span}}

\def\codim{{\rm codim}}
\def\crd{\dim_{{\rm CR}}}
\def\crc{{\rm codim_{CR}}}

\def\phi{\varphi}
\def\eps{\varepsilon}
\def\d{\partial}
\def\a{\alpha}
\def\b{\beta}
\def\g{\gamma}
\def\G{\Gamma}
\def\D{\Delta}
\def\Om{\Omega}
\def\k{\kappa}
\def\l{\lambda}
\def\La{\Lambda}
\def\z{{\bar z}}
\def\w{{\bar w}}
\def\Z{{\1Z}}
\def\t{\tau}
\def\th{\theta}

\emergencystretch15pt
\frenchspacing

\newtheorem{Thm}{Theorem}[section]
\newtheorem{Cor}[Thm]{Corollary}
\newtheorem{Pro}[Thm]{Proposition}
\newtheorem{Lem}[Thm]{Lemma}

\theoremstyle{definition}\newtheorem{Def}[Thm]{Definition}

\theoremstyle{remark}
\newtheorem{Rem}[Thm]{Remark}
\newtheorem{Exa}[Thm]{Example}
\newtheorem{Exs}[Thm]{Examples}

\def\bl{\begin{Lem}}
\def\el{\end{Lem}}
\def\bp{\begin{Pro}}
\def\ep{\end{Pro}}
\def\bt{\begin{Thm}}
\def\et{\end{Thm}}
\def\bc{\begin{Cor}}
\def\ec{\end{Cor}}
\def\bd{\begin{Def}}
\def\ed{\end{Def}}
\def\br{\begin{Rem}}
\def\er{\end{Rem}}
\def\be{\begin{Exa}}
\def\ee{\end{Exa}}
\def\bpf{\begin{proof}}
\def\epf{\end{proof}}
\def\ben{\begin{enumerate}}
\def\een{\end{enumerate}}
\def\beq{\begin{equation}}
\def\eeq{\end{equation}}

\begin{abstract}
We propose a new class of geometric invariants called {\em jet vanishing orders},
and use them to establish a new selection algorithm
in the Kohn's construction of subelliptic multipliers for special domains in dimension $3$,
inspired by the work of  Y.-T. Siu~\cite{S10}.
In particular, we obtain effective termination of our selection algorithm
with explicit bounds both for the steps of the algorithm
and the order of subellipticity in the corresponding subelliptic estimates.
Our procedure possesses additional features of certain stability
under high order perturbations, due to deferring the step of taking 
radicals to the very end, see Remark~\ref{feat} for more details.

We further illustrate by examples the sharpness in our technical results
(in Section~\ref{4})
and demonstrate the complete procedure
for arbitrary high order perturbations 
of the Catlin-D'Angelo example \cite{CD10}  in Section~\ref{cd-ex}.

Our techniques here may be of broader interest
for more general PDE systems,
in the light of the recent program initiated
by the breakthrough paper
of Y.-T. Siu~\cite{S17}.

\end{abstract}

\maketitle

\tableofcontents

\section{Introduction}
In his seminal paper \cite{K79},
J.J.~Kohn invented a purely algebraic construction
of ideals of subelliptic multipliers for the $\bar\d$-Neumann problem.
The goal of this note is to propose a new class of geometric invariants,
called {\em jet vanishing orders},
that permits us to obtain a fine-grained control of the
effectiveness in the Kohn's construction procedure of subelliptic multipliers
for the so-called {\em special domains} of finite D'Angelo type
\cite{D79, D82} in $\C^3$.

Since this paper is dedicated to Professor Ngaiming Mok,
we would like to mention a striking parallel 
between the Kohn's {\em subelliptic multipliers ideals} and
the {\em varieties of minimal rational tangents} 
(VMRT) pioneered by N.~Mok and J.-M.~Hwang
\cite{HM98}.
Both theories connect {\em global} PDE or algebraic-geometric structures
with {\em local} differential-geometric objects
that can be treated by local geometric and analytic methods,
subsequently leading to important consequences
of global nature.
For more details on the VMRT,
see the articles and lecture notes
by Hwang and Mok
\cite{Mo08, H14}.
Another important parallel is historical.
The VMRT theory was developed 
to study deformation rigidity problems
originated from the celebrated
Kodaira-Spencer's work
on deformation of complex structures,
with
an important ingredient
coming from
D.~Spencer's program
to generalize 
the Hodge theory of
 harmonic integrals.
And it was the same program
that led Spencer to formulate the
$\bar\d$-Neumann problem,
for which Kohn invented his multiplier ideals approach
to tackle the problem of local regularity
(see \cite{K-etal04} for a more detailed account).

Note that Kohn's original procedure in \cite{K79}
gives no effective bound on the order of subellipticity $\eps$
in subelliptic estimates,
as illustrated by 
examples 
of G.~Heier~\cite{He08} and
 D.W.~Catlin and J.P.~D'Angelo \cite{CD10},
see \S\ref{eff} and \S\ref{cd-ex} below.
On the other hand, 
Y.-T. Siu \cite{S10, S17} obtained a new effective procedure
for special domains
and outlined an extension
of the special domain approach 
to general real-analytic and smooth cases.
Also note that for the special domains
of the so-called {\em triangular form} in $\C^n$,
a different effective procedure
in Kohn's algorithm
was given by Catlin-D'Angelo \cite[Section~5]{CD10}.
Triangular systems were first introduced by D'Angelo~\cite{D95} under the term {\em regular coordinate domains}, where also an effective procedure
for obtaining subelliptic estimates was provided.
See also 
D.W.~Catlin and J.-S.~Cho \cite{CC08} and
T.V.~Khanh and G.~Zampieri \cite{KhZa14} for subelliptic estimates 
for triangular systems by a different method,
and
A.~Basyrov, A.C.~Nicoara and the second author
\cite{BNZ17}
for an approximation
of general
smooth pseudoconvex boundaries
by triangular systems of
sums of squares matching
the Catlin multitype.
We next mention that A.C.~Nicoara \cite{N14}
proposed a construction for the termination of the Kohn algorithm in the real-analytic case with an indication of the ingredients needed for the effectivity.
We further refer to 
\cite{K79, K84, D95, DK99, S01, S02, K04, S05, S09, CD10, S10, S17}
for more profound and extensive discussion
of subelliptic multipliers
(see also \cite{Ch06} for an algebraic approach),
as well as surveys
\cite{Si91, BS99, FS01, M03, St06, MV15}
and books \cite{Mr66, FK72, Tr80, 
Ho90, 
D93, ChS01, O02, Z08, St10, Ta11, Ha14, O15} 
 for the $\bar\d$-Neumann problem in broader context.

In relation to Kohn's foundational work
on the subelliptic multipliers,
we would like to mention a remarkable new development in the field
due to Siu \cite{S17} providing 
new techniques of generating multipliers for general systems of partial differential equations,
including, as a special case,
a new procedure even
for the case of the $\bar\d$-Neumann problem.

Finally we mention the related
important Nadel's multiplier ideal sheaves \cite{N90}
that were originally motivated by the ones
defined by Kohn \cite{K79} and are
in some sense dual to them.
See \cite[\S4]{S01}, \cite[1.4.6]{S09} for more detailed discussions
of the relation between both types of multipliers,
and the expository articles and books 
\cite{S01, S02, De01, L04, S05, S09} for further information.

\subsection{Kohn's algorithm for special domains
and subelliptic estimates}\Label{js}
We shall consider so-called {\em special domains}
as introduced by Kohn in \cite{K79},
defined locally near the origin by holomorphic functions
in the first $n$ variables:
\beq\Label{om}
	\Omega:= \{
	\Re(z_{n+1})+\sum_{j=1}^m |F_j(z_1,\ldots, z_n)|^2<0
	\} \subset \C^{n+1},
\eeq
where $F_1,\ldots,F_m$ is any collection of holomorphic functions
vanishing at the origin in $\cn$
(and defined in its neighborhood).

Understanding Kohn's multipliers
for the special domains
is important, on the one hand, due to 
their link with local analytic and algebraic geometry,
and on the other hand, 
through their connection with more general cases
via D'Angelo's construction
of associated families of holomorphic ideals
\cite[Chapter~3]{D93}
and Siu's program
relating special domains approach
with general real-analytic and smooth cases
\cite[II.3, II. 4]{S10}.

Next recall the {\em Kohn's algorithm} or, more precisely,
its holomorphic variant for
special domains \eqref{om} (corresponding to $q=1$ in the notation of \cite{K79}
and to algebraic geometric formulation in \cite[2.9.4]{S17}).
For any set $S$ of germs of holomorphic functions at $0$ in $\cn$,
consider the ideal $J(S)$
generated by all function germs $g$ satisfying
$$
	df_1\wedge\ldots\wedge df_n =
	g\, dz_1\wedge \ldots\wedge dz_n , \quad
f_1,\ldots, f_n\in S,
$$
i.e.\ by all {\em Jacobian determinants}
\beq\Label{jac0}
	g =
	\frac{\d(f_1,\ldots,f_n)}{\d(z_1,\ldots,z_n)} =
\det\left(\frac{\d f_i}{\d z_j}\right)
.
\eeq
Then the Kohn's {\em $0$th multiplier ideal} for \eqref{om} is defined as the radical
\beq\Label{1st}
	I_0:=\sqrt{J(S)},
	\quad
	S: = \{F_1, \ldots, F_m\},
\eeq
and for any $k>0$, the Kohn's {\em $k$th multiplier ideal} is
defined inductively by
\beq\Label{kth}
	I_k:= \sqrt{J(S\cup I_{k-1})}.
\eeq
The obtained increasing sequence of ideals
$
	I_0 \subset I_1 \subset ...
$
is said to {\em terminate at $k$} if $I_k$ contains the unit $1$.

Kohn proved in \cite{K79} that
the necessary and sufficient condition for the ideal termination
for special domains 
(in fact, for his more general algorithm applied
to domains with {\em real-analytic} boundaries)
 is the finiteness of the D'Angelo type of $\d\Omega$ at $0$.
Note that Kohn's algorithm is actually defined for
domains with arbitrary smooth boundaries,
in which generality,
however, the corresponding termination
remains a major open problem.
For a special domain \eqref{om}, the type equals $2\t(S)$,
where $S$ is given by \eqref{1st} and
\beq\Label{type}
	\t(S):= \sup_\g \inf_{f\in S} \frac{\nu(f\circ \g)}{\nu(\g)}
\eeq
is the type of the set $S$ of holomorphic function germs,
where the supremum is taken over all
nonconstant
germs of holomorphic curves
$
	\g\colon (\C,0)\to (\cn,0),
$
and $\nu(g)$ denotes the vanishing order of a smooth vector function $g$ at $0$,
given by the lowest order of its nonvanishing partial derivative at $0$.
We refer to \cite{D82, D93} for more detailed 
discussions of finite type
and \cite{M92b, BS92, FIK96, D17, MM17, Z17}
for equivalent characterizations in 
various particular cases, see also \cite{FLZ14}.



It follows from Kohn's work \cite{K79}
that 
ideal termination 
at a boundary point
$p\in \d\Omega$
leads to 
a {\em subelliptic estimate of some order of subellipticity $\eps>0$}  
at that point, namely
\beq\Label{subel}
	\|u\|_\eps^2
	\le
	C(
		\|\bar\d u\|^2
		+ \|\bar\d^* u\|^2
		+\|u\|^2
	),
\eeq
where 
$\|\cdot\|_\eps$ 
and 
$\|\cdot\|$
are
respectively the tangential Sobolev norm
of the (fractional) order $\eps$
and
the standard $L^2$ norm on $\Omega$,
$u$ is any $(0,1)$ form
in the domain of the adjoint operator $\bar\d^*$
(with respect to the standard $L^2$ product on $\Omega$),
smooth up to the boundary
and with compact support in
a fixed neighborhood $U$ of $p$ in the closure $\1\Omega$,
such that both $U$ and the real constant $C>0$
are independent of $u$,
see \cite[Definition 1.11]{K79} for more details.
(Note that Kohn obtained
his results for $(p,q)$ forms with arbitrary $p,q$.)
We mention that,
even though the estimate \eqref{subel}
depends on the choice of local coordinates
(or the metric used in the $L^2$ product),
the subellipticity property
(i.e.\ the existence of an estimate \eqref{subel}
for given $\eps$)
is invariant \cite{Sw72, Ce08, CSt09}.

In the same paper,
Kohn further proved that
his algorithm does terminate
for pseudoconvex domains with real-analytic boundaries of finite type,
leading to subelliptic estimates,
based on a result of 
K.~Diederich and J.E.~Forn\ae ss~\cite{DF78}
(see also E.~Bedford and J.E.~Forn\ae ss~\cite{BF81},
Siu~\cite[Part~IV]{S10} 
and Kohn~\cite{K10}).
For pseudoconvex domains with  smooth boundaries of 
D'Angelo finite type,
subelliptic estimates are due to
D.W.~Catlin~\cite{C87} by a different method,
whereas the termination
of Kohn's algorithm in that generality
remains a major open problem
(see e.g.\ the first problem in \cite[Section~14]{DK99}).
Note that the finite type condition 
is also necessary for subelliptic estimates
due to the work of P. Greiner~\cite{G74} in $\C^2$ and
D.W.~Catlin~\cite{C83} in $\C^n$.

It is crucial to point out 
the particularly remarkable feature
of Kohn's procedure
that allows us to 
study functional-analytic 
estimates such as \eqref{subel}
with purely geometric methods
applied to ideals of holomorphic functions.
%
%
%
%
Consequently, our approach here
is accessible to any geometer
without familiarity
with the functional-analytic aspect
of the problem.

\subsection{Multipliers and effectivity}\Label{eff}
Recall \cite{K79} that
a germ at $p\in\d\Omega$ 
of a smooth function on the closure $\1\Omega$
is a {\em subelliptic multiplier
with order of subellipticity $\eps$} if,
for some representative $f$ of the germ,
 the estimate similar to \eqref{subel}
with $\|u\|_\eps^2$ replaced by $\|fu\|_\eps^2$
holds under the same assumptions, namely
\beq\Label{subel-f}
	\|fu\|_\eps^2
	\le
	C(
		\|\bar\d u\|^2
		+ \|\bar\d^* u\|^2
		+\|u\|^2
	),
\eeq
where both $U$ and $C$ may depend on $f$.
It follows from Kohn's work
\cite[Section~4 and 7]{K79}
that the ideals $I_k$
as defined in \eqref{kth}
consist of subelliptic multipliers,
i.e.\ satisfy \eqref{subel-f}
{\em with some order of subellipticity $\eps$}
that may depend on the actual multipliers.
More precisely,
Kohn showed that 
taking
a Jacobian determinant in \eqref{jac0} reduces the
order of subellipticity by $1/2$,
whereas taking a root of order $s$
in the radical
reduces it by $1/s$.
{\em It is the last step
where an effective
control of $\eps$ may get lost,}
as it is a priori not clear
what root order is required
in order to obtain the full radical.
In fact, examples of Heier \cite[Section~1.1]{He08}
and Catlin-D'Angelo
\cite[Proposition~4.4]{CD10}
illustrate precisely 
that, i.e.\
the lack of control of $\eps$
as the parameter $K$ in the example
goes to infinity,
whereas the type remains bounded
(see Section~\ref{cd-ex} below).
See also Siu~\cite[4.1]{S17}
for a more elaborate and detailed explanation of this 
important phenomenon.
A perturbation of Catlin-D'Angelo's example
is treated in Section~\ref{cd-ex} below.

In order to regain the effectiveness,
one needs to restrict the 
the number of steps and the
root orders allowed in the radicals
in terms of only the type and dimension,
leading to  an
{\em effective Kohn algorithm}  
in the terminology of Siu \cite[2.6]{S17}).
In addition, it is desired to have
algorithmic selection rules
for constructing sequences of the actual multipliers ending with $1$.
Our goal here is to obtain a fine effectiveness control in terms of our new invariants (that are in turn controlled by the type)
and a new selection procedure for the multipliers leading to an effective termination.
In particular, we are going to the extreme to avoid 
taking radicals until the very last step,
see Remark~\ref{feat} below for details.

%

\subsection{Applications and further directions}
Among notable applications of subelliptic estimates \eqref{subel},
J.J.~Kohn and L.~ Nirenberg \cite{KN65}
proved that the latter imply 
{\em local boundary regularity} of the Kohn's
solution of the $\bar\d$-Neumann problem
$\bar\d u = f$,
i.e. the solution $u$ is smooth at those boundary points
where $f$ is.
See \cite{K72, DK99} and the references therein
for more detailed discussions of this 
and many other applications.



Another important application
of subelliptic estimates,
specifically demonstrating the 
{\em importance
of the effectivity},
is a lower bound on the Bergman metric
directly related to the order of subellipticity,
due to J.D.~McNeal \cite{M92}.
The effective control in subelliptic estimates
also plays important role
in the construction of peak functions
by J.E.~Forn\ae ss and J.D.~McNeal 
\cite{FM94}
and in other results.

It should be noted that 
applications of Kohn's algorithm
and its effectiveness
are not limited to subelliptic estimates.
In their remarkable paper \cite{DF79},
K.~Diederich and J.E.~Forn\ae ss
discovered a direct way of using Kohn's multipliers
to construct so-called plurisubharmonic ``bumping''
functions for arbitrary domains with
real-analytic boundaries of finite type.
The bumping functions
are subsequently applied
in the same paper
for Kobayashi metric estimates
and H\"older regularity of proper 
holomorphic maps.
However, the exponents
in the crucial estimates
could not be
effectively controled
(in terms of the type and dimension)
due to some steps
involving Lojasiewicz inequality.
Here effective 
procedures for generating multipliers
would allow for more explicit 
quantitative conclusions.

In a more recent work,
Kohn \cite{K00, K02, K04}
has demonstrated some new
use of subelliptic multipliers
by relating them to certain new
microlocal subelliptic multipliers
on real hypersurfaces,
and establishing
 hypoellipticity
of the $\square_b$ and
 the $\bar\d$-Neumann operator.
 In \cite{K05} Kohn introduced
 an analogue of his theory of subelliptic multipliers
 for 
 new classes of
 differential operators.
  See also
L.~Baracco~\cite{Ba15}
 and
 L.~Baracco, S.~Pinton and G.~Zampieri~\cite{BaPZa15}
for recent related results.

Inspired by Kohn's original subelliptic multipliers,
analogous notion of
multiplier ideals for the compactness estimate
were studied by 
M.~\c Celik~\cite{Ce08},
E.J.~Straube~\cite{St08},
in their joint work~\cite{CSt09}
and by M.~\c Celik and Y.E.~Zeytuncu~\cite{CZ17}.
Finally, 
we mention results
by D.~Chakrabarti and M.-C.~Shaw~\cite{CS11}
 connecting properties of
the Kohn's solution for individual domains 
with corresponding properties for 
to their products,
with some of their results recently used by
\mbox{X.-X.~Chen} and S.K.~Donaldson \cite{ChD13}
to study rigidity properties of complex structures.

\subsection{Main results}
Our main new tool
is the following invariant number
associated to a set $S$ of germs
of holomorphic functions at $0$ in $\cn$,
a germ of analytic subvariety $(V,0)$ in $\cn$
and an integer $k$:
\beq\Label{k-type}
	\t^k_V(S):= \sup_\g \inf_{f\in S} \frac{\nu(j^kf\circ \g)}{\nu(\g)},
	\quad k\ge 0,
\eeq
where $j^kf = (\d^\a f)_{|\a|\le k}$ is the vector of all partial
derivatives up to order $k$,
and the supremum is taken over the set of all
nonconstant germs of holomorphic maps
$\g\colon (\C,0)\to (V,0)$,
and the vanishing order $\nu$ is as above.
We call $\t^k_V$ the {\em $k$-jet type of $S$ along $V$}.
In particular, for $k=0$, we obtain the D'Angelo type \eqref{type}.
It is easy to see that the $k$-jet types form a non-increasing sequence
for $k=0,1,\ldots$, and become equal to $0$ whenever $k$
is greater or equal the minimum vanishing order of a germ in $S$.
Note that, even though the space of curves in \eqref{k-type} is infinite-dimensional,
the computation of the $k$-jet type in $\C^2$ can be reduced
to certain finite number of curves
by a result of J.D.~McNeal and A.~N\'emethi \cite{MN05}
(see also \cite{
He08, LT08} for further results in this direction).

Since the effective termination is well-known
if at least one function $F_j$ has order $1$
(see e.g.\ \cite{CD10}),
we shall assume that all functions $F_j$ vanish
of order $\ge 2$ at $0$.
The following is a simplified version (with rougher bounds) of the main result of the paper:

\bt\Label{main}
In the context of Kohn's algorithm for special domains \eqref{om} in $\C^3$
of D'Angelo finite type $\le 2T$ at $0$,
there is an effective algorithmic construction of a sequence
of multipliers
\beq\Label{l}
	f_1,\ldots, f_l,
	\quad
	l\le T(T-1)+4,
\eeq
with $f_l=1$,
based on the invariants \eqref{k-type}.
Every $f_j$ is obtained by taking
a Jacobian determinant of a linear combination of elements in the set
\beq\Label{mult-set}
	\{F_1, \ldots, F_m, f_1,\ldots, f_{j-1}\},
\eeq
except the two multipliers $f_{l-2}= z_1$, $f_{l-1}= z_2$
that are in the radical of the ideal $I:=(f_1, \ldots, f_{l-3})$
with a root order $s\le T^2(T-1)^3$, i.e.\
$z_1^s, z_2^s\in I$.
\et

\br\Label{feat}
We would like to point out the following particular features
in our multiplier construction in comparison to other known procedures.
From the three different procedures in Kohn's construction,
namely {\em (1) taking Jacobian determinants, (2) forming ideals and (3) taking the radicals},
only the first one is used in our process to obtain a pair of
multipliers generating an ideal of finite (effectively controlled) multiplicity.
Only then the ideal is formed and the radical is taken
to obtain the linear multipliers, leading to the termination.

In particular, our procedure possesses certain
stability under high order perturbations,
as illustrated in \S\ref{cd-ex}.
Such stability is generally not available
in procedures based on taking radicals.
For instance, in the ring of germs of holomorphic functions
at $0$ in $\C^2_{z,w}$,
$f=w$ is in the radical of
the ideal generated by $f^2 = w^2$,
but any perturbation $g=f^2 + z^m = w^2 + z^m$
with $m$ odd generates an ideal $I_m$
having no other germs in its radical,
i.e.\ $\sqrt{I_m} = I_m$
(since the zero variety of $I_m$
is locally irreducible at $0$).
This problem does not occur in our construction
because we only take radicals from ideals
of finite codimension,
where any germ vanishing at $0$
is in the radical.
\er

\bc\Label{main-cor}
For special domains \eqref{om} of finite type $\le 2T$ at $0$ in $\C^3$,
a subelliptic estimate \eqref{subel} at $0$
holds
with the order of subellipticity
$$
    \eps \ge \frac1{2^{l-1}s} \ge \frac1{2^{T(T-1)+3} \, T^2(T-1)^3},
$$
where $l$ is the number of multipliers in \eqref{l} and
 $s$ is the radical root order in Theorem~\ref{main}.
\ec

\bpf
It follows from Kohn's work \cite[Sections 4 and 7]{K79},
that in the context of Theorem~\ref{main}, 
a Jacobian determinant of the $F_j$
has the order of subellipticity $\ge 1/4$
(see \cite[(4.29, 4.64)]{K79} and \cite[2.9.2]{S17}),
and taking any further Jacobian determinant reduces the
order of subellipticity by $1/2$
(see \cite[(4.42, 4.64)]{K79}),
whereas taking a root of order $s$
reduces it by $1/s$ (see \cite[4.36]{K79}).
Applying these calculations to the construction
in Theorem~\ref{main}, we obtain the desired conclusion.
\epf

Note that we obtained more refined bounds than those given in Theorem~\ref{main}
and Corollary~\ref{main-cor}
in terms of our new invariants \eqref{k-type}.
Also we illustrate in Section~\ref{cd-ex} how our procedure can be applied
to the Catlin-D'Angelo's example as well as its higher order perturbations.

\subsection{Overview of our procedure}
Following \cite{S10, S17}, we call the functions $F_j$ in \eqref{om}
{\em pre-multipliers}, to distinguish them from
the {\em multipliers} obtained via the algorithm.
It is important to emphasize that the pre-multipliers
are only used inside the Jacobian determinants \eqref{jac0}
but are never added to the ideals directly.

On a large scale, there are {\em $3$ major steps},
 each reducing the dimension of the variety defined by the multipliers
(resembling Kohn's original algorithm for real-analytic hypersurfaces
 \cite{K79}).

The {\em first step} consists of constructing
the first multiplier $f_1$ as
Jacobian determinant
from the pre-multipliers
with effectively bounded vanishing order at $0$.
Our method gives a bound of at most $\le T(T-1)$.
In fact, a finer bound is given in terms of the $k$-jet types
(as defined by \eqref{k-type}) of the pre-multipliers
along their zero curves, see \eqref{det-order}.
Note that other bounds are known from the work of
D'Angelo \cite{D82,
D93}, Siu \cite{S10, S17}
and Nicoara \cite{N12}.
Any known effective bound $d$ can be used at this step to proceed
with our construction.

The {\em second major step}
aims to reduce the dimension of the variety
$$
	V:= \{f_1=0\}
$$
from $1$ to $0$.
Our method here is based
on a sequence of {\em minor steps}
constructing new multipliers
that gradually reduce the jet order $k\ge 0$
for which the $k$-jet type along $V$ can be effectively bounded.
The construction starts with $k=d\le T(T-1)$
(the vanishing order of $f_1$),
and at every minor step,
the order $k$ is reduced from $k_j$ to $k_{j+1}<k_j$,
where the new $k_{j+1}$-jet type along $V$
gets an effective bound equal to the previous
bound for the $k_j$-jet type plus at most
$(k_j - k_{j+1})(T-1)$.
In other words, lowering the jet order by a number $N$
increases the type bound by $N(T-1)$.
In fact, the method gives a finer bound
with $T$ replaced by the vanishing order
of pre-multipliers only along curves in $V$.
Based on the configuration of the $k$-jet types,
multiple possible choices for $k_j$ are available
(see Section~\ref{4} for details),
where fewer Jacobian determinant iterations
can be traded for possibly higher vanishing order bound and vice versa.
At the end of this major step,
we obtain a multiplier $f_{l-3}$
whose type (i.e.\ the ($0$-jet) type as defined for $k=0$ in \eqref{k-type}) along $V$
is effectively bounded by an estimate
no worse than $d(T-1)$
(where $d\le T(T-1)$ is the
vanishing order of $f_1$).
The proof is based on the core technical
results in Section \ref{4}
with examples given
illustrating the sharpness of the assumptions.

Finally, the {\em third major step}
consists of taking the coordinate functions
$f_{l-2}:=z_1$ and $f_{l-1}:=z_2$
in the radical of the finite type ideal
$I(f_1, f_{l-3})$,
and subsequently
their Jacobian determinant $f_l=1$.
The corresponding radical root order
can be effectively bounded
by
$$
	td \le d^2(T-1) \le T^2(T-1)^3
$$
 in terms of
 the vanishing order $d$ of $f_1$
and the type $t\le d(T-1)$ of $f_{l-3}$ along $V$.
See Lemma~\ref{powers} 
for the bound $td$.

\section{Vanishing orders, contact orders and jets}

\subsection{Normalised vanishing order}
We write $f \colon (\cn, 0)\to \C$ for a germ at $0$
of a holomorphic function in $\cn$
(without specifying the value $f(0)$),
and
$$
	\nu(f) := \min \{ |\a| : \d_z^\a f (0) \ne 0\} \in \N\cup \{\infty\},
\quad \N= \{0, 1, \ldots\}
$$
for the {\em vanishing order} of $f$, 
also called {\em multiplicity} in the literature,
 (the minimum is $\infty$ if the set is empty), where
$$
	\a=(\a_1,\ldots, \a_n)\in \4N^n,
	\quad
	|\a| := \a_1 + \ldots + \a_n,
$$
is a multiindex and
$$
	\d_z^\a = \d_{z^1}^{\a_1} \ldots \d_{z_n}^{\a_n}
$$
is the corresponding partial derivative with respect to chosen coordinates
$(z_1,\ldots, z_n)\in \cn$.
Clearly $\nu(f)$ does not depend on the choice of local holomorphic coordinates.

More generally, for any set $S$ of germs of holomorphic maps
(on the same $\cn$),
define its vanishing order to be the minimum
vanishing order for its elements:
$$
	\nu(S) := \min \{ \nu(f) : f\in S\}.
$$
It is clear that $\nu(S)=\nu(I(S))$,
where $I$ is the ideal generated by $S$.
In particular, for any holomorphic map germ
$$
	f = (f_1,\ldots,f_m)\colon (\cn,0)\to \C^m,
$$
define its vanishing order to be the vanishing order
of the set of its components,
i.e.\
$$
	\nu(f) := \min_{j=1,\ldots,m} \nu(f_j).
$$
It is again easy to see that $\nu(f)$
also does not depend on the choice of local holomorphic
coordinates in $\C^m$ in a neighborhood of $f(0)$.

Now following D'Angelo~\cite{D82, D93}, for every germ of a
holomorphic map
$$
	\g\not\equiv 0 \colon (\C^m,0) \to (\cn, 0),
	\quad
	f\colon (\cn,0) \to \C^\ell,
$$
define its {\em normalized vanishing order of $f$ along $\g$}
by
\beq\Label{normalized}
	\nu_\g (f) := \frac{\nu(f\circ \g)}{\nu(\g)}.
\eeq
In case $m=1$ (considered here),
the normalized vanishing order
is invariant under singular parameter changes
$\g \mapsto \g\circ \phi$,
where $\phi\colon (\C,0)\to (\C,0)$
is any nonconstant germ of a holomorphic map.
Similarly, for a set $S$ of holomorphic map germs, define
its normalized vanishing order along $\g$ by
$$
	\nu_\g(S) := \min \{ \nu_\g(f) : f\in S\},
$$
and again $\nu_\g(S) =\nu_\g(I(S))$
for the ideal $I(S)$ generated by $S$.

Note that in general, the normalized vanishing
order may not be an integer,
and, in fact, can be any rational number $p/q>1$
as the example
$$
	\g(t):= (t^p, t^q),
\quad	f(z,w) := z,
	\quad \nu_\g(f) = p/q,
$$
shows.

It is easy to see that given $f$, one has $\nu_\g(f)= \nu(f)$
for a generic linear map $\g$.
In general, $\nu_\g(f)$ can only become larger:

\bl\Label{g}
The normalized vanishing order of $f$
along any map $\g$
is always greater or equal than the vanishing order:
\beq
	\nu_\g(f) \ge \nu(f).
\eeq
\el

\bpf
Expanding into a power series $f= \sum f_\a z^\a$
with $|\a| \ge \nu(f)$,
and substituting $\g$,
we conclude
$\nu(f\circ \g) \ge \nu(f) \, \nu(\g)$ as desired.
\epf

\subsection{Contact order and finite type}

Let $S$ be a set of germs of holomorphic functions
$f\colon (\cn,0)\to \C$.
Then in view of Lemma~\ref{g} (and the remark before it),
the vanishing order satisfies
$$
	\nu(S) = \min_\g \nu_\g(S).
$$
At the opposite end, define (as in \cite{D82}
and \cite[2.3.2, Definition~9]{D93})  the {\em contact order}
or {\em the type} of $S$
by
$$
	\t(S) := \sup_\g \nu_\g(S) = \sup_\g \min_{f\in S} \nu_\g(f).
$$
(Note that the minimum in the right-hand side is always achieved,
since the denominator in \eqref{normalized} is fixed.)
We say that the set $S$ is of {\em finite type} $T=\t(S)$
if the latter number is finite.
Since $\nu_\g(S) = \nu_\g(I(S))$
for the ideal $I(S)$ generated by $S$,
we also have
$$
	\nu(S) = \nu(I(S)),
	\quad
	\t(S) = \t(I(S)).
$$

\subsection{Contact order along subvarieties.}
Let $f \colon (\cn, 0)\to \C^m$ be a germ of a
holomorphic map as before
and $V\subset \cn$ any complex-analytic subvariety passing through $0$.
Then the {\em contact order} of $f$ along $V$ is defined by
\beq\Label{contact}
\t_V(f):= \sup_\g \nu_\g (f)
\eeq
where the supremum is taken over all non-constant germs of holomorphic maps $\g\colon (\C,0)\to (V,0)$.

In particular, if $V$ is (a germ of) an irreducible curve at $0$,
the right-hand side in \eqref{contact} is independent of $\g$,
since in this case,
any two nonzero germs $\g\colon (\C,0)\to (V,0)$
are related by a sequence of singular reparametrizations.
In general, the contact order along $V$
is the maximum contact order along
the irreducible components of $V$ at $0$.
%
%
%

\subsection{Jet vanishing orders}
The main new tool in this paper is the following notion
of jet vanishing order.

For every integer $k\ge 0$, and a germ of a holomorphic map
$f\colon (\cn,0)\to \C^m$,
consider its $k$-jet $j^kf$,
which in local coordinates can be regarded
as a germ of a holomorphic map
$$F=j^k f = (\d^\a_z f)_{|\a|\le k}\colon (\cn,0)\to \C^N,$$
(for suitable integer $N$ dependent on $n$, $m$ and $k$)
given by all partial derivatives of the components of $f$ up to order $k$.
Then define the {\em $k$-jet vanishing order of $f$} by
\beq\Label{jet-order}
	\nu^k(f) : = \nu(j^kf) = \min_{|\a|\le k} \nu(\d^\a_z f),
\eeq
which is also equal to $\max( \nu(f) -k, 0)$.

Further, define
the {\em $k$-jet normalized vanishing order of $f$ along
a nonzero germ $\g\colon (\C^m,0)\to(\cn,0)$} by
$$
	\nu^k_\g(f):= \nu_\g(j^k f)= \min_{|\a|\le k} \nu_\g(\d^\a_z f),
$$
(which are the minimum vanishing orders
of the partial derivatives up to order $k$ along $\g$),
and the {\em $k$-jet contact order of $f$ along
a subvariety $V$} by
\beq
\t^k_V(f):=  \t_V(j^k f) = \sup_\g \nu^k_\g(f),
\eeq
where as in \eqref{contact}, the supremum is taken over all
non-constant germs of holomorphic maps $\g\colon (\C,0)\to (V,0)$.
Note that again, this definition is independent of holomorphic local coordinates,
and we have the monotonicity:
\beq
	\nu^{k_1}(f) \ge \nu^{k_2}(f),
	\quad \nu^{k_1}_\g(f) \ge \nu^{k_2}_\g(f),
	\quad \t^{k_1}_V(f) \ge \t^{k_2}_V(f),
	\quad k_1 \le k_2.
\eeq
As consequence of \eqref{jet-order}
and the definitions, we have the stabilisation property
\beq
	\nu^k(f) = \nu^k_\g(f) =  \t^k_V(f) =0,
	\quad k \ge \nu(f),
\eeq
for any $\g$ and any $V$.

Similarly, for any set $S$ of germs of holomorphic maps
$f\colon (\cn,0)\to \C^m$,
define the numbers
$$
	\nu^k(S) : = \nu (j^k S),
	\quad
	\nu^k_\g(S) : = \nu_\g (j^k S),	
	\quad
	\t^k(S) : = \t(j^k S),
	\quad
	\t^k_V(S) : = \t_V(j^k S),	
$$
that, in fact, depend only on the ideal generated by $S$,
where $j^k S$ denotes the set of all $k$-jets of elements in $S$.

\section{Control of jet vanishing orders for Jacobian determinants}
\Label{4}

This section is the technical core of the paper.
Our goal is to obtain fine control of how the jet vanishing orders
along curves change under taking Jacobian determinants.
The main idea is to have a control of certain ``good'' terms
in the multiplier expansion and to avoid possible cancellations
with other terms.
This is achieved via certain technical conditions
comparing jet vanishing orders for different jet orders.
We also illustrate by examples that our technical conditions are sharp.

For a holomorphic curve germ
$\g\not\equiv 0 \colon (\C,0)\to (\C^2,0)$,
we can always change coordinates $(z,w)\in\C^2$ to achieve
\beq\Label{ab}
    \g = (\a, \b), \quad \nu(\a)>\nu(\b)\ge 1.
\eeq

We first give a
comparison condition between jet vanishing orders
$\nu_\g^{k-1}(F)$ and $\nu_\g^k(F)$
 of a function germ $F$ along $\g$
to guarantee that
the minimum vanishing order along $\g$
among partial derivatives of $F$
is achieved for its transversal derivatives.

\bl\Label{transversal}
Given $\g$ satisfying \eqref{ab} and a germ of a holomorphic map
$$
	F\colon (\C^2, 0)\to \C,
$$
assume that the jet vanishing orders of $F$ along $\g$ satisfy
\beq\Label{smaller0}
	\nu_\g^{k-1}(F) > \nu_\g^k(F) + 1
\eeq
for some $k\ge 1$.

Then
\beq\Label{min-transv}
	\nu_\g^k(F)
	= \nu_\g(\partial_z^{k} F),
\eeq
i.e.\ the minimum vanishing order among all partial derivatives
of order $k$ is achieved for the $k$th derivative
transversal to (the image of) $\g$:
$$
	\nu_\g(\partial_z^{k} F) =
    \min_{a+b \le k} \nu_\g(\partial_z^a\partial_w^{b} F).
$$
\el

\bpf
Assume on the contrary that
$$
	\nu_\g(\partial_z^a\partial_w^{b} F)
	< \nu_\g(\partial_z^{k} F)
$$
for some $a<k$ and $b\le k-a$.
We can choose $a$ and $b$ such that the minimum vanishing order is achieved, i.e.\
\beq\Label{a-chosen}
	\nu_\g^k(F)
	=\nu_\g(\partial_z^a\partial_w^{b} F)
	< \nu_\g(\partial_z^{k} F)
	.
\eeq
In view of \eqref{smaller0}, we must have the top order $a+b=k$ here,
in particular, $b=k-a\ge 1$.
Then
\beq\Label{smaller}
	\nu(\partial_z^a\partial_w^{b} F(\gamma))
	\le \nu (		
		\partial^{a+1}_z\partial^{b-1}_w F(\gamma)
	)
	< \nu\left(
		\partial^{a+1}_z\partial^{b-1}_w F(\gamma)\frac{\a'}{\b'}
	\right),
\eeq
as we have assumed $\nu(\a) > \nu(\b)$.
Differentiating in the parameter of $\g$, we obtain
$$
	\left(
		\partial_z^a\partial_w^{b-1} F(\gamma)
	\right)'
	=\b'\left(\partial^{a+1}_z\partial^{b-1}_w F(\gamma)\frac{\a'}{\b'}
+\partial_z^a\partial_w^{b} F(\gamma)\right),
$$
from which, using \eqref{smaller}, we conclude
\beq\Label{der-order}
	\nu\left(
		(
			\partial_z^a\partial_w^{b-1} F(\gamma)
		)'
	\right)
	=\nu(\partial_z^a\partial_w^{b} F(\gamma))+\nu(\b)-1.
\eeq
Since our definition of the $k$-jet vanishing order implies
$$
	\nu\left(
		(
			\partial_z^a\partial_w^{b-1} F(\gamma)
		)'
	\right)
	\geq \nu(\b)
	\, \nu_\g^{k-1}(F)-1,
$$
we have in view of \eqref{a-chosen} and \eqref{der-order},
\beq\Label{stuck}
	\nu(\b)
	\, \nu_\g^k(F)
	= \nu(\partial_z^a\partial_w^{b} F(\gamma))=\nu\left(
		(
			\partial_z^a\partial_w^{b-1} F(\gamma)
		)'
	\right)-\nu(\b)+1
	\geq \nu(\b)
		\, \nu_\g^{k-1}(F)-\nu(\b).
\eeq
On the other hand, by our assumption \eqref{smaller0}, we have
$$
	\nu(\b)
	\, \nu_\g^{k-1}(F)-\nu(\b)
	>\nu(\b) \, \nu_\g^k(F),
$$
which contradicts \eqref{stuck}
completing the proof.
\epf

\be
For
$$
    \g(t) = (0, t),
    \quad
    F(z,w) = w^2 + zw^2,
$$
compute
$$
    \nu^0_\g(F)= \nu_\g(F)  = 2,
    \quad
    \nu_\g(\d_z F) = 2,
    \quad
    \nu_\g(\d_w F) = 1,
    \quad
    \nu^1_\g(F) = 1.
$$
Hence \eqref{smaller0} is violated
and
 the conclusion of Lemma~\ref{transversal} fails.
This shows the sharpness of our assumption \eqref{smaller0}.
\ee

The following is our main technical result.
Again, we need to assume a comparison condition
between different jet vanishing orders
to guarantee that the terms from lower order jets
do not cancel with the terms providing the needed vanishing order control.
Here we use different letters $F$ and $\phi$ for the function germs
to emphasize their different roles.
The role of $F$ is to provide the jet vanishing orders $\nu^k_\g(F)$,
whereas for $\phi$, only the usual vanishing order $\nu_\g(\phi)$ is used.

\begin{Lem}\Label{main-tech}
Given germs of holomorphic maps
$$
	\g\not\equiv 0 \colon (\C,0)\to (\C^2,0),
	\quad
	F,\phi\colon (\C^2, 0)\to \C,
$$
suppose that
\beq\Label{phi-2}
	\nu(\phi)\ge 2,
\eeq
 and for some $k\ge 1$,
\begin{equation}\Label{main1}
	\nu_\g^{k-1}(F) >
	\nu^k_\g(F) + \nu_\g(\phi)-1,
\end{equation}
where $\nu_\g^k$ is the $k$-jet vanishing order along $\g$.
Then the Jacobian determinant
$$
	G := \det
\begin{pmatrix}
\partial F\cr
\partial\phi
\end{pmatrix}
$$
satisfies
$$
	\nu_\g^{k-1}(G)=\nu_\g^k(F)+ \nu_\g(\phi) -1.
$$
\end{Lem}

\begin{proof}
By definition, we have
$$
	G =
	\partial_zF \, \partial_w\phi
	-\partial_wF \, \partial_z\phi,
$$
and hence
\begin{equation}\label{main2}
	\partial_z^a\partial_w^b G
	=\partial_z^{a+1} \partial_w^bF
		\, \partial_w\phi
	-\partial_z^a\partial_w^{b+1}F
		\, \partial_z\phi
	+ {\rm error~ terms},
	\quad
	a+b=k-1,
\end{equation}
where the error terms are bilinear expressions in $j^{k-1}F$
and $j^k\phi$.

After a coordinate change if necessary, we may assume that
the assumptions of Lemma~\ref{transversal} are satisfied,
 in particular
$$
	\nu(\g) = \nu(\b)
$$
 and
$\b\not\equiv 0$.

Next our assumption \eqref{main1}
implies that the (normalized) vanishing order of the error terms in \eqref{main2}
along $\g$
 is strictly bigger than
$$
	\nu_\g^k(F)+ \nu_\g(\phi) -1.
$$
Therefore, to prove the lemma, it is enough to show that
\beq\Label{to-show}
	\min_{a+b=k-1}
	\left(
		\nu_\g(
		\partial_z^{a+1}\partial_w^bF
			\, \partial_w\phi
		- \partial_z^a\partial_w^{b+1}F
			\, \partial_z\phi)
	\right)
	=\nu_\g^k(F)+ \nu_\g(\phi) -1.
\eeq

Since
$$
	\left(\partial_z^a\partial_w^bF(\g)\right)'
	= \partial_z^{a+1}\partial_w^bF(\gamma)\alpha'+\partial^a_z\partial_w^{b+1}F(\gamma)\beta'
$$
and
$$(\phi (\g))'=\partial_z\phi(\gamma)\alpha'+\partial_w\phi(\gamma)\beta',$$
 substituting these into the first two terms in the right-hand side of \eqref{main2}, we obtain
\beq\Label{rewr}
	\partial_z^{a+1}\partial_w^bF(\gamma)
		\, \partial_w\phi(\gamma)
	- \partial_z^a\partial_w^{b+1}F(\gamma)
		\, \partial_z\phi(\gamma)
	=\partial_z^{a+1}\partial_w^bF(\gamma)\frac{(\phi (\g))'}{\beta'}-\frac{\left(\partial_z^a\partial_w^{b}F(\gamma)\right)'}
{\beta'}\partial_z\phi(\gamma).
\eeq
Since for $a+b=k-1$,
\begin{eqnarray*}
\nu \left(\frac{\left(\partial_z^a\partial_w^{b}F (\g)\right)'}
{\beta'}\partial_z\phi(\gamma)\right)&=&\nu((\partial_z^a\partial_w^{b}F(\gamma))')-\nu(\beta')+\nu(\partial_z\phi(\gamma))\\
&=&\nu((\partial_z^a\partial_w^{b}F(\gamma)))-1-\nu(\beta)+1+\nu(\partial_z\phi(\gamma))\\
&=&
	\nu(\beta)\left(
		\nu_\g(\partial_z^a\partial_w^{b}F)-1
		+ \nu_\g(\partial_z\phi)
	\right)\\
&\geq &
	\nu(\beta) \left(
		\, \nu_\g^{k-1}(F) - 1 + \nu_\g(\partial_z\phi)
	\right)\\
\end{eqnarray*}
we obtain by \eqref{main1}, \eqref{phi-2} and Lemma~\ref{g},
\beq\Label{bigger}
	\nu \left(\frac{\left(\partial_z^a\partial_w^{b}F(\gamma)\right)'}
	{\beta'}\partial_z\phi(\gamma)\right)>\nu(\beta)
	(\nu_\g^k(F)+ \nu_\g(\phi) - 1),
	\quad a+b \le k-1.
\eeq

On the other hand, we have
\begin{eqnarray*}
	\nu\left(\partial_z^{a+1}\partial_w^bF(\gamma)\frac{(\phi (\g))'}{\beta'}\right)
	&=&
	\nu(\partial_z^{a+1}\partial_w^bF(\gamma))
	+\nu((\phi (\g))')-\nu({\beta'})\\
	&=&
	\nu(\partial_z^{a+1}\partial_w^bF(\gamma))+\nu(\phi (\g))-\nu({\beta})\\
	&=& \nu(\beta)
		\left(\nu_\g(\partial_z^{a+1}\partial_w^bF)+\nu_\g(\phi)-1)\right),
\end{eqnarray*}
where in view of \eqref{min-transv},
the minimum of the vanishing order of the last expression for $a+b=k-1$ is
achieved for $(a,b)=(k-1,0)$ and equals to
\beq\Label{min}
	\nu(\beta)
		\left(
			\nu^k_\g(F)+ \nu_\g(\phi) -1
		\right).
\eeq
Together with \eqref{bigger},
this implies that
the minimum for $a+b=k-1$ of the right-hand side in \eqref{rewr}
equals \eqref{min}.
This gives the desired relation \eqref{to-show},
completing the proof.
\end{proof}

%
%

\be
Let
$$
    F(z,w) = z^k + z^{k-1} w^s, \quad
    \phi(z,w) = w^l + a z w,
    \quad
    \g(t) = (0, t),
    \quad l\ge 2.
$$
Then
$$
    \nu^{k-1}_\g(F)=s, \quad
    \nu^k_\g(F) = 0, \quad
    \nu_\g(\phi) = l.
$$
In particular, our comparison assumption \eqref{main1} is violated
if and only if $s\le l-1$.
Then for the Jacobian determinant $G$ as Lemma~\ref{main-tech}, we have
$$
    G = \det
    \begin{pmatrix}
        k z^{k-1}  + (k-1) z^{k-2} w^s &   s z^{k-1} w^{s-1} \\
        a w & l w^{l-1}  + a(l-s) z w
    \end{pmatrix}
$$
and hence for $s=l-1$ and $a=lk/s$, the terms with $z^{k-1} w^{l-1}$ cancel and we obtain
$$
    \nu^{k-1}_\g(G) = 
    	\min
		\big(
			\nu^{k-1}_\g(z^{k-2}w^{s+l-1}), \nu^{k-1}_\g(z^{k-1}w^{s+1})
		\big)
    = \min( s+l-1, s+1) = l,
$$
failing the conclusion \eqref{main1} of Lemma~\ref{main-tech}.
Hence our assumption \eqref{main1} is sharp.
\ee

As first direct consequence in combination with Lemma~\ref{g},
 we obtain an estimate for the (total) vanishing order
of the Jacobian determinant:
\bc
	Under the assumptions of Lemma~\ref{main-tech},
	$$
		\nu(G) \le \nu^k_\g(F) + \nu_\g(\phi) + k - 1.
	$$
\ec

As next immediate consequence of Lemma~\ref{main-tech}, we obtain:

\bc\Label{down}
Given germs of holomorphic maps
$$
	\g\not\equiv 0 \colon (\C,0)\to (\C^2,0),
	\quad
	F,\phi\colon (\C^2, 0)\to \C,
$$
suppose that
\beq\Label{phi-2'}
	\nu(\phi)\ge 2.
\eeq
Then
$$
	\nu^{k-1}_\g(G) \le \nu_\g^k(F) + \nu_\g(\phi) -1,
$$
where either $G=F$ or
$
G=
\det
\begin{pmatrix}
\partial F\cr
\partial\phi
\end{pmatrix}
$.
\ec

Repeatedly applying
Corollary~\ref{down},
we obtain:

\bc\Label{multi}
Given germs of holomorphic maps
$$
	\g\not\equiv 0 \colon (\C,0)\to (\C^2,0),
	\quad
	f,\phi\colon (\C^2, 0)\to \C,
$$
suppose that
$	\nu(\phi)\ge 2$,
and define $f_{j+1}$ inductively by
$$
f_1 := f,
\quad
f_{j+1} := \det
\begin{pmatrix}
\partial f_{j}\cr
\partial\phi
\end{pmatrix},\quad j= 1,2,\ldots.
$$
Then
for every $k\in \{0, \ldots, \nu(f)\}$,
there exists
$j\in \{1, \ldots, \nu(f)-k+1\}$ such that
$$
	\nu^k_\g(f_{j})\leq (\nu(f)-k)(\nu_\g(\phi)-1).
$$
In particular, for $k=0$,
$$
	\nu_\g(f_{j})\leq \nu(f)(\nu_\g(\phi)-1)
$$
holds for some
$j\in \{1, \ldots, \nu(f) + 1\}$.
\ec

\begin{proof}
Setting
$
	m:= \nu(f)
$,
we obtain from the definition of the $m$-jet vanishing order that
$$
	\nu_\g^m(f_1) = \nu_\g^m(f) = 0.
$$
%
%
Then Corollary~\ref{down} implies
that there exist $j\in\{1,2\}$
such that
$$\nu_\g^{m-1}(f_j)\leq \nu_\g(\phi)-1.$$
Next let $k=m-1$ and
repeat the argument for $F=f_j$,
to conclude that
there exist $j\in\{1,2,3\}$ such that
$$\nu_\g^{m-2}(f_j)\leq 2 (\nu_\g(\phi)-1).$$
We repeat this process to show inductively
that for any positive integer $\ell\leq m$,
there exists $j\in\{1,\ldots,\ell+1\}$
such that
$$\nu_\g^\ell(f_j)\leq (m-\ell) (\nu_\g(\phi)-1)$$
as desired.
\end{proof}

\section{Selection algorithm and Proof of Theorem~\ref{main}}
We shall write $\6C$ for the set of
nonzero germs of holomorphic maps
$(\C,0)\to (\C^2,0)$.
Let $S$ be a set of holomorphic function germs
$(\C^2,0)\to \C$ with minimal vanishing order
$$
	m:=\nu(S) \ge 2,
$$
and the
 finite type
$$
	T:=\t(S).
$$

\subsection{Step 1; selecting a multiplier with minimal vanishing order}
Choose any $f_0=f\in S$
with minimal (total) vanishing order
$$
	\nu(f) = \nu(S) = m.
$$
Then choose any $\g\in\6C$
in the zero set of $f$ (i.e.\ $\nu_\g(f) = \infty$).
Finally choose any $\phi\in S$ with
$$
	\nu_\g(\phi) \le T
$$
and
set
$$
	\mu := \min \{ k-1 + \nu_\g^k(f) : \nu_\g^{k-1}(f) > \nu_\g^k(f) + \nu_\g(\phi) - 1,\,  1\le k \le m \}.
$$
Since $\nu_\g(f) =\infty$, it is easy to see that
$$
	 \mu \le k_0-1+\nu_\gamma^{k_0}(f) \le  k_0-1+(m-k_0)(\nu_\g(\phi) -1)\leq (m-1)(\nu_\g(\phi) -1),
$$
where
$$k_0:=\max\{k:\nu_\g^{k-1}(f) > \nu_\g^k(f) + \nu_\g(\phi) - 1,\,  1\le k \le m \}.$$
Then the multiplier
$$
	f_1 :=
	\det
	\begin{pmatrix}
		\d f\cr
	\d \phi
	\end{pmatrix}.
$$
satisfies for some $k$,
$$
	k + \nu_\g^{k}(f_1) \le \mu + \nu_\g(\phi) - 1
$$
in view of
Lemma~\ref{main-tech}.
In particular, it follows that the vanishing order
\beq\Label{det-order}
	\nu(f_1) \le \mu + \nu_\g(\phi) - 1
		\le m (\nu_\g(\phi) - 1) \le m (T-1).
\eeq
Thus we have found a multiplier $f_1\in J(S)$
(as defined in \S\ref{js})
whose vanishing order is bounded by any
of the above estimates.

\br
The main outcome of this step is to construct a multiplier $f_1$
with controlled vanishing order.
Our method gives an explicit estimate for this order.
Note that the initial function (pre-multiplier) $f_0\in S$ is only used in this step.
Once a multiplier $f_1$ is constructed, it will be used in the sequel
for the curve selection and the computation of jet vanishing orders.
\er

\subsection{Step 2; selecting a multiplier with effectively bounded order along a zero curve of another multiplier}
Now choose any multiplier $f_1$ with
vanishing order
$$
	\nu(f_1)\le d \le m(T-1).
$$
Note that Step 1 yields such $f_1$,
however, any other $f_1$ with a sharper estimate $d$
can be chosen.
Let $\g_i\in \6C$ be germs of curves parametrizing
irreducible components of the zero set of $f_1$.
Choose $\phi\in \span(S)$ with
$$
	\nu_{\g_i}(\phi) \le T,
$$
which can be any generic linear combination
of elements in $S$.
Then Corollary~\ref{multi} implies that
$$
	\nu^{m-1}_{\g_i}(f_j) \le T-1,
$$
where $j\in\{1,2\}$
depending on $i$,
and
$$
	f_2 :=
	\det
	\begin{pmatrix}
		\d f_1\cr
	\d \phi
	\end{pmatrix}.
$$

\bl\Label{41}
Let $\Phi$, $S_0$ be two sets of germs of holomorphic maps $(\C^2,0)\to\C$
and define inductively
$$
	S_{j+1} := S_j \cup
\left\{
		\det
\begin{pmatrix}
\partial f\cr
\partial \phi
\end{pmatrix}
: f\in S_j\cup \Phi, \, \phi\in \Phi
\right\},
\quad
j\ge 0.
$$
Write
$$
	d:= \nu (S_0)
$$
for the minimum vanishing order of functions in $S_0$
and
assume that the type
$$
	\t(\Phi) := \sup_\g \nu_\g(\Phi) = \sup_\g \min_{\phi\in \Phi} \nu_\g(\phi)
$$
is finite. Assume further that
$$\nu(\Phi)\geq 2.$$
Then for $k< d$,  the $k$-jet type of $S_j$ satisfies
$$
	\t^k(S_{d-k}) \le (d-k)(\t(\Phi)-1).
$$
In particular, for $k=0$, we obtain
$$
	\t(S_d) \le d (\t(\Phi)-1).
$$
\el

\bpf
Choose $f\in S_0$ with
$$
	\nu(f) = d
$$
and let $\g$ parametrize a curve in $\mathbb{C}^2$.
By the definition of type,
there exists $\phi\in \Phi$
with
$$
	\nu_\g(\phi) \le \t(\Phi).
$$
Then Corollary~\ref{multi} implies that there exists a sequence $F_k \in S_k$ such that
$$\nu_\gamma^{d-k}(F_k)\le (d-k)(\nu_\gamma(\phi)-1),$$
which completes the proof.
\epf

\subsection{Step 3; selecting a pair of multpliers with
effectively bounded multiplicity and taking the radical}
Let $\Phi$ be the set of pre-multipliers and let
$$S_0:=\left\{
		\det
\begin{pmatrix}
\partial f\cr
\partial \phi
\end{pmatrix}
: f, \phi\in \Phi
\right\}.$$
Then Step 2 implies that there exists an integer $d\leq m(T-1)$ such that
$$\tau(S_d)\leq d(T-1).$$
Choose $F\in S_0$ such that
$$\nu(F)=d.$$
Then by the definition of $\t(S_d)$, for each irreducible component $C_j$ of $\{F=0\}$,
there exists $F_j\in S_d$ such that
$$\nu_{ C_j}(F_j)\leq d(T-1).$$
Choose a generic linear combination $\widetilde F$ of $\{F_j\}$.
Then $\widetilde F$ is a multiplier satisfying
\beq\Label{ft}
    \nu_{C_j}(\widetilde F)\leq d(T-1)
    \text{ for all } j.
\eeq

\begin{Lem}\Label{powers}
Let $I = (F,\widetilde F)$ be the ideal generated by 
$F$ and $\widetilde F$.
Let $d:=\nu(F)$ be the vanishing order
and $t:=\tau_{\{F=0\}}(\widetilde F)$
the contact order of $\2F$ along the 
zero curve of $F$
(i.e.\ the maximal vanishing order
along an irreducible component).
Then the multiplicity 
\beq\Label{mult}
	D(I):=\dim\left(\mathbb{C}\{z,w\}/ I\right)
\eeq
 satisfies $D(I) \le dt$.
\el

\bpf
The proof can be obtained by following the arguments of D'Angelo's proof of Theorem~2.7 in  \cite{D82},
more precisely,
by combining \cite[(2.12)]{D82} with \cite[Lemma~2.11]{D82}.
\epf

\br
A pair $(F,\2F)$ satisfying the assumptions of Lemma~\ref{powers}
can be an called an {\em effective regular sequence}.
Recall that a {\em regular sequence} 
(see e.g.\ \cite[2.2.3, Definition 6]{D93})
is any sequence $(f_1, \ldots, f_s)$ in a local ring $R$,
if each $f_{j+1}$ is a non-zero-divisor in the quotient
$R/(f_1,\ldots, f_j)$ for any $j=1,\ldots, s$
(for $j=1$, the quotient needs to be interpreted
as $R$ itself).
In our case, the additional {\em effectiveness}
of the regular sequence condition
comes from the estimates of the two orders --
the vanishing order of $F$
and the contact order of $\2F$ along the zero curve $C$ of $F$.
Note that for the conclusion of
Lemma~\ref{powers} to hold, we need
to estimate the {\em contact order}
(the maximum order over irreducible components of $C$) 
for $\2F$ rather than its vanishing order along $C$ 
(the minimum order over the irreducible components).
\er

\bc\Label{K}
Under the assumptions of Lemma~\ref{powers},
for $a:=dt$,
one has $h^a\in I$ for any germ of holomorphic function 
$h$ with $h(0)=0$,
i.e.\ $I$ contains the $a$-th power of the maximal ideal.
\ec

\bpf
The statement follows directly from
the inequality $K(I)\le D(I)$ in
\cite[Theorem~2.7]{D82},
 see also \cite[2.1.6, Page 57]{D93},
 where $K(I)$ is the minimal power 
 of the maximal ideal
 contained in $I$
 and $D(I)$ is the multiplicity. 
 \epf

\bpf[Proof of Theorem~\ref{main}]

By Corollary~\ref{K},
 we have $z^a, w^a$ as multipliers,
where
$$a= dt, \quad d\le T(T-1), \quad t \le d(T-1)$$
implying, in particular,
$$
a\le T^2(T-1)^3.
$$
Taking the $a$th roots $z,w$ of $z^a, w^a$, and taking their Jacobian determinant, we
obtain the last multiplier $1$ leading to the algorithm termination.
Furthermore, taking generic linear combinations at each step,
we obtain a single sequence $f_1,\ldots, f_l$ of multipliers
with $f_{l-3}=\2F$ satisfying \eqref{ft},
and hence satisfying the desired properties
as in the statement of Theorem~\ref{main}.
This
completes the proof.
\epf

\section{Perturbations of Heier's and Catlin-D'Angelo's examples}
\Label{cd-ex}
Heier \cite{He08} and
Catlin and D'Angelo \cite{CD10}
gave examples of special domains in
$\C^3$,
where the original Kohn's procedure \cite{K79}
of taking full radicals at every step
does not lead to an effective estimate (in terms of the type and the dimension) for
the order of subellipticity in subelliptic estimate.
The main reason is the lack of control of the root order in the radical.
Heier's example is
$$
	\Re z_3 + |z_1^3 + z_1 z_2^K|^2 + |z_2|^2,
	\quad
	K\ge 2,
$$
where the set of pre-multipliers is
$S =\{  z_1^3 + z_1 z_2^K,  z_2 \}$ and a
calculation of Kohn's multiplier ideals
 yields, in the notation of Section~\ref{js}:
$$
	I_0 = \sqrt{J(S)} = J(S) = (3z_1^2 + z_2^K),
	\quad
	J(S\cup I_0) = (z_1, z_2^{K}),
	\quad
	I_1 = \sqrt{J(S\cup I_0)} = (z_1, z_2),
$$
where the last radical
requires taking elements of the root order $K$
that can be arbitrarily high
in comparison with the D'Angelo type $6$.
As consequence, the corresponding order of subellipticity 
$\eps$ in the subelliptic estimate \eqref{subel}
obtained this way is not effectively controlled by the type.
Note that since $z_2\in S$, $z_1\in J(S\cup I_0)$,
it is easy to regain the effectivity
by taking another Jacobian determinant
instead of the radical to 
obtain 
$$
	1\in J(S\cup J(S \cup I_0)),
$$
leading to an effective subelliptic estimate 
with $\eps = 1/8$.
In particular, the step of taking radicals
for this example can be avoided completely,
due to the presence of the linear pre-multiplier $z_2\in S$.

A similar lack of effectivity phenomenon is exhibited in
the Catlin-D'Angelo's example
as explained in \cite[Section 4]{CD10} and \cite[Section 4]{S17}
(where, however, it is not possible to avoid
taking radicals when there is no pre-multiplier of 
vanishing order $1$).
Furthermore, 
the same lack of control with the same argument also applies
to higher order perturbations of that example
(that may not be of a triangular form).

More precisely,
consider perturbations $\Omega\subset\C^3$
of the Catlin-D'Angelo's example
 given by
$$
    \Re z_3 + |F_1(z_1,z_2)|^2 + |F_2(z_1,z_2)|^2 < 0,
$$
$$
    F_1 = z_1^M + O(L),
    \quad
    F_2 = z_2^N + z_1^K z_2 + O(L),
    \quad
    K > M\ge 2,
    \quad
     N\ge 3,
$$
where we use the notation
$$
    O(L) = O(|(z_1,z_2)|^L).
$$
Then for $L$ sufficiently large, the explicit calculations in \cite[Proposition~4.4]{CD10} and \cite[Section 4.1]{S17}
show that the root order required for the radical $I_1$ of $J_1:= J(S\cup I_0)$ (in the notation of \S\ref{js})
is at least $K$, i.e.\ $(I_1)^{K-1}\not\subset J_1$.
As consequence, the order of subellipticity obtained this way is not effectively controlled by the type.

We now illustrate how our selection procedure modifying the original Kohn's algorithm applies to this case.
For simplicity, we shall assume 
$L$ sufficiently large but effectively bounded from below by the type 
$T=\max(M,N)$ (i.e.\ $L\ge \Phi(T)$ for suitable function $\Phi$
that can be directly computed),
and
$K$ being sufficiently large (when the above non-effectivity occurs).
The first multiplier generating $J_0:=J(S)$ is
\beq\Label{f1}
    f_1 :=
    \det
    \begin{pmatrix}
    \d F_1\cr
    \d F_2
    \end{pmatrix}
    =
    \det
    \begin{pmatrix}
    M z_1^{M-1} & 0\cr
    K z_1^{K-1} z_2 & N z_2^{N-1} + z_1^K 
    \end{pmatrix}
    =
    MN z_1^{M-1} z_2^{N-1} + M z_1^{K+M-1}
    + O(L-1),
\eeq
where the perturbation error estimate $L-1$ is only rough
for the sake of simplicity.
Then a monodromy argument implies that 
the zero set of $f_1$ is the union of the (possibly reducible) curves 
\beq\Label{curves}
    z_1^{M-1}=O(L-N), \quad
    N z_2^{N-1} + z_1^K = O(L-M).
\eeq

For any parametrisation $\g$ of an irreducible component of the first curve,
we take  in our Step 2
$$
    \phi = F_2, \quad \nu_\g(\phi) = N,
$$
and compute
$$
	f_2 =
	\det
	\begin{pmatrix}
		\d f_1\cr
	\d \phi
	\end{pmatrix}
    \sim
	   z_1^{M-2} z_2^{2(N-1)}
    + O(L-2)
    \mod z_1^{K-1},
$$
where we write $\sim$ for the equality up to a constant factor.
The corresponding jet vanishing orders of $f_1$ are
$$
    \nu^0_\g(f_1), \ldots, \nu^{M-2}_\g(f_1) \ge \frac{L-N}{M-1},
    \quad
    \nu^{M-1}_\g(f_1) = N-1,
    \quad
     \nu^{M}_\g(f_1) = N-2,
     \quad
   \ldots,
   \quad
   \nu^{M+N-2}_\g(f_1) = 0.
$$
Then our comparison condition \eqref{main2}
holds for $k=M-1$, and hence Lemma~\ref{main-tech}
gives the desired control of the next lower jet vanishing order
$$
    \nu^{M-2}_\g(f_2) = 2(N-1).
$$
Next, continuing as in Step 2 in the previous section,
we obtain the sequence of multipliers
$$
    f_j = \det
    	\begin{pmatrix}
		\d f_{j-1}\cr
	\d \phi
	\end{pmatrix}
    \sim
    z_1^{M-j} z_2^{j(N-1)}
    + O(L-j)
     \mod z_1^{K-j+1},
    \quad
    j=1,2,\ldots, M,
$$
where the last multiplier
$$
    f_M = z_2^{M(N-1)} + O(L-M) \mod z_1^{K-M+1}
$$
 has the finite order $M(N-1)$ along $\g$.

Similarly, for any curve $\g$ in the variety defined by the second equation \eqref{curves},
we take in our Step 1 the other pre-multiplier
$$
    \psi = F_1 = z_1^M + O(L),
    \quad
    \nu_\g(\psi) = M.
$$
Then following our Step 2 procedure, we obtain the sequence of multipliers
$$
    g_1 = f_1,
    \quad
    g_j = \det
    	\begin{pmatrix}
		\d g_{j-1}\cr
	\d \psi
	\end{pmatrix}
    \sim
    z_1^{j(M-1)} z_2^{N-j}
    + O(L-j),
    \quad
    j=2,\ldots, N,
$$
where the last multiplier
$$
    g_N = z_1^{N(M-1)}
    + O(L-N)
$$
 has the finite order $N(M-1)$ along $\g$.

Finally, in Step 3, we take
$$
    F=f_1 = MN z_1^{M-1} z_2^{N-1} + M z_1^{K+M-1}
    + O(L-1),
$$
$$
    \2F=f_M + g_N = z_2^{M(N-1)} +  z_1^{N(M-1)} + O(L-\max(M,N)) \mod z_1^{K-M+1}
$$
 such that
$$
    d= \nu(F)=M+N-2,
    \quad
    t=\nu_{F=0}(\2F) = \max(N(M-1), M(N-1))
$$
and the ideal generated by $F$ and $\2F$ has the multiplicity
$\le dt$
which bounds the root order in the radical,
to obtain the multipliers $z_1$ and $z_2$,
and hence their Jacobian determinant,
leading to the desired termination.
Furthermore, the number of steps and the root order when taking the radical,
and hence the order of subellipticity in Corollary~\ref{main-cor}
are explicitly controlled in terms of the type $2T=2\max(M,N)$.

\bigskip

{\bf Acknowledgements}.
We would like to thank Professor Joseph J. Kohn for 
suggesting the topic and encouragement,
and Professor John P. D'Angelo
for careful reading of the manuscript
and numerous helpful remarks as well as
many inspiring discussions.
We are also grateful to Professor Yum-Tong Siu for his interest
and encouragement
 and for kindly sending us his paper \cite{S17} containing
 the  most
recent account on the subject and setting future development 
perspectives.
We finally would like to thank the anonymous referee for careful reading
and numerous helpful suggestions.


\begin{thebibliography}{11111111}

%

\bibitem[Ba15]{Ba15}
Baracco, L.;
A multiplier condition for hypoellipticity of complex vector fields 
with optimal loss of derivatives. 
{\em J. Math. Anal. Appl.} {\bf 423} (2015), no. 1, 318--325. 



\bibitem[BaPZa15]{BaPZa15}
Baracco, L.; Pinton, S.; Zampieri, G.
Hypoellipticity of the Kohn-Laplacian $\square_b$ and of the 
$\bar\d$-Neumann problem by means of subelliptic multipliers.
{\em Math. Ann.} {\bf 362} (2015), no. 3-4, 887--901.

\bibitem[BNZ17]{BNZ17}
Basyrov,~A.;
Nicoara,~A.C.;
Zaitsev,~D.
	Sums of squares in pseudoconvex hypersurfaces
	and torsion phenomena for Catlin's boundary systems.
Preprint 2017.

\bibitem[BF81]{BF81}
Bedford, E.; Forn\ae ss, J.E.
Complex manifolds in pseudoconvex boundaries. 
{\em Duke Math. J.} {\bf 48} (1981), no. 1, 279--288. 

\bibitem[BSt92]{BS92}
 Boas, H.P.; Straube, E.J. 
On equality of line type and variety type of real hypersurfaces in $\C^n$. 
{\em J. Geom. Anal.} {\bf 2} (1992), no. 2, 95--98. 


\bibitem[BSt99]{BS99}
 Boas, H.P.; Straube, E.J. 
 Global regularity of the $\bar\d$-Neumann problem: a survey of the 
 $L^2$-Sobolev theory. Several complex variables (Berkeley, CA, 1995--1996), 79--111, Math. Sci. Res. Inst. Publ., 37, Cambridge Univ. Press, Cambridge, 1999.


\bibitem[C83]{C83}  D. Catlin. 
Necessary conditions for subellipticity of the $\bar\d$-Neumann problem. {\em Ann. of Math.} (2) {\bf 117} (1983), no. 1, 147--171.
%

\bibitem[C87]{C87} D. Catlin. 
Subelliptic estimates for the $\bar\d$-Neumann problem on pseudoconvex domains. {\em Ann. of Math. (2)}, 
{\bf 126} (1): 131--191, (1987).


\bibitem[CC08]{CC08}
Catlin, D.W.; Cho, J.S.
Sharp Estimates for the 
$\bar\d$-Neumann Problem on Regular Coordinate Domains.
Preprint 2008.
{\tt https://arxiv.org/abs/0811.0830} 

\bibitem[CD10]{CD10}
Catlin, D.W.; D'Angelo, J.P.
Subelliptic estimates. Complex analysis, 75--94,
Trends Math., Birkh\"auser/Springer Basel AG, Basel, 2010.


\bibitem[Ce08]{Ce08}
\c Celik, M.
Contributions to the compactness theory of the 
$\bar\d$-Neumann operator, Ph. D. dissertation,
Texas A\&M University, May 2008.

\bibitem[CeSt09]{CSt09}
\c Celik, M.; Straube, E.J.
Observations regarding compactness in the $\bar\d$-Neumann problem, 
{\em Complex Var. Elliptic Equ.} {\bf 54} (2009), no. 3-4, 173--186.

\bibitem[CeZ17]{CZ17}
\c Celik, M.;  Zeytuncu, E.Z.
Obstructions for Compactness of Hankel Operators: Compactness Multipliers. Preprint 2017.
To appear in the Illinois Journal of Mathematics.
{\sf https://arxiv.org/abs/1611.06377}



\bibitem[ChS11]{CS11}
Chakrabarti, D.; Shaw, M.-C.
The Cauchy-Riemann equations on product domains.
{\em Math. Ann.} {\bf 349} (2011), no. 4, 977--998. 



\bibitem[ChD13]{ChD13}
Chen, X.-X.; Donaldson, S.K.
Volume estimates for K\"ahler-Einstein metrics and rigidity of complex structures.
{\em J. Differential Geom.} {\bf 93} (2013), no. 2, 191--201. 

\bibitem[CS01]{ChS01}
 Chen, S.-C.; Shaw, M.-C. 
 Partial differential equations in several complex variables. 
 AMS/IP Studies in Advanced Mathematics, 19. American Mathematical Society, Providence, RI; International Press, Boston, MA, 2001.


\bibitem[Ch06]{Ch06}
Cho, J.-S.
An algebraic version of subelliptic multipliers.
Michigan Math. J. {\bf 54} (2006), no. 2, 411--426.

%

\bibitem[D79]{D79} 
D'Angelo,~J.P.
Finite type conditions for real hypersurfaces. 
{\em J. Differential Geom.} {\bf 14} (1979), no. 1, 59--66 (1980).


\bibitem[D82]{D82} 
D'Angelo,~J.P. 
Real hypersurfaces, orders of contact, and applications.
{\em Ann.~of Math.~(2)}, {\bf 115} (3), 615--637, 1982.


\bibitem [D93]{D93} D'Angelo, J.P.
Several complex variables and the geometry of real hypersurfaces.
Studies in Advanced Mathematics. CRC Press, Boca Raton, FL, 1993.

\bibitem [D95]{D95} 
D'Angelo, J.P.
Finite type conditions and subelliptic estimates. Modern methods in complex analysis (Princeton, NJ, 1992), 63--78, Ann. of Math. Stud., 137, Princeton Univ. Press, Princeton, NJ, 1995.

\bibitem [D17]{D17} 
D'Angelo, J.P.
A remark on finite type conditions.
{\em J. Geom. Anal.} (2018) 28: 2602. 
{\tt https://doi.org/10.1007/s12220-017-9921-1};
{\tt https://arxiv.org/abs/1708.07794}

\bibitem [DK99]{DK99}
D'Angelo, J.P.; Kohn, J.J.
Subelliptic estimates and finite type. Several complex variables (Berkeley, CA, 1995--1996), 199--232,
Math. Sci. Res. Inst. Publ., {\bf 37}, Cambridge Univ. Press, Cambridge, 1999.



\bibitem [De01]{De01}
Demailly, J.-P. 
Multiplier ideal sheaves and analytic methods in algebraic geometry. School on Vanishing Theorems and Effective Results in Algebraic Geometry (Trieste, 2000), 1--148, ICTP Lect. Notes, 6, 
Abdus Salam Int. Cent. Theoret. Phys., Trieste, 2001.

\bibitem [DiF78]{DF78}
Diederich, K.; Forn\ae ss, J.E. 
Pseudoconvex domains with real-analytic boundary. 
{\em Ann. Math. (2)} {\bf 107} (1978), no. 2, 371--384. 

\bibitem [DiF79]{DF79}
Diederich, K.; Forn\ae ss, J.E. 
Proper holomorphic maps onto pseudoconvex domains with real-analytic boundary. 
{\em Ann. of Math. (2)} {\bf 110} (1979), no. 3, 575--592. 



\bibitem [FK72]{FK72}
Folland, G.B.; Kohn, J.J.
The Neumann problem for the Cauchy-Riemann complex. 
Annals of Mathematics Studies, No. 75. Princeton University Press, Princeton, N.J.; University of Tokyo Press, Tokyo, 1972.

\bibitem [FLZ14]{FLZ14}
Forn\ae ss, J.E.; 
Lee, L.; 
Zhang, Y.
Formal complex curves in real smooth hypersurfaces.
{\em Illinois J. Math.} {\bf 58} (2014), no. 1, 1--10. 

\bibitem [FM94]{FM94}
Forn\ae ss, J.E.; 
McNeal, J.D.
A construction of peak functions on some finite type domains. 
{\em Amer. J. Math.} {\bf 116} (1994), no. 3, 737--755. 


\bibitem[FIK96]{FIK96} 
Fu, S.; Isaev, A.V.; Krantz, S.G.
Finite type conditions on Reinhardt domains. 
{\em Complex Variables Theory Appl.} {\bf 31} (1996), no. 4, 357--363. 


\bibitem [FS01]{FS01}
Fu, S.; Straube, E.J. 
Compactness in the $\bar\d$-Neumann problem. Complex analysis and geometry (Columbus, OH, 1999), 141Ð160, Ohio State Univ. Math. Res. Inst. Publ., 9, de Gruyter, Berlin, 2001.


\bibitem [Gr74]{G74}
Greiner, P.
Subelliptic estimates for the $\bar\d$-Neumann problem in $\C^2$. 
{\em J. Differential Geometry} {\bf 9} (1974), 239--250. 

\bibitem [Ha14]{Ha14}
Haslinger, F.
The $\bar\d$-Neumann problem and Schr\"odinger operators. 
De Gruyter Expositions in Mathematics, 59. De Gruyter, Berlin, 2014.

\bibitem [He08]{He08}
Heier, G. 
Finite type and the effective Nullstellensatz. {\em Comm. Algebra} {\bf 36} (2008), no. 8, 2947--2957.

\bibitem [Ho90]{Ho90}
 H\"ormander, L. 
 An introduction to complex analysis in several variables. Third edition. North-Holland Mathematical Library, 7. North-Holland Publishing Co., Amsterdam, 1990.


\bibitem [H14]{H14}
 Hwang, J.-M.
Mori geometry meets Cartan geometry: Varieties of minimal rational tangents.
To appear in Proceedings of ICM2014.
{\tt https://arxiv.org/abs/1501.04720}

\bibitem [HM98]{HM98}
 Hwang, J.-M.; Mok, N. 
 Rigidity of irreducible Hermitian symmetric spaces of the compact type under
K\"ahler deformation, 
 {\em Invent. math.} {\bf 131} (1998), 393--418.

 
%

\bibitem [KhZa14]{KhZa14}
Khanh, T.V.; Zampieri, G. 
Precise subelliptic estimates for a class of special domains. 
{\em J. Anal. Math.} {\bf 123} (2014), 171--181.




%


%


\bibitem[K72]{K72}
Kohn, J. J.
Boundary behavior of $\bar\d$ on weakly pseudo-convex manifolds of dimension two. Collection of articles dedicated to S. S. Chern and D. C. Spencer on their sixtieth birthdays. 
{\em J. Differential Geometry} {\bf 6} (1972), 523--542.


\bibitem[K79]{K79}
Kohn, J. J.
Subellipticity of the $\bar\d$-Neumann problem on pseudo-convex domains: sufficient conditions. {\em Acta Math.} {\bf 142} (1979), no. 1-2, 79--122.

\bibitem[K84]{K84}
Kohn, J. J.
A survey of the $\bar\d$-Neumann problem, pp. 137--145 in
Complex analysis of several variables (Madison, WI, 1982), 
edited by Y.-T. Siu,
Proc. Symp. Pure Math. 41, Amer. Math. Soc., Providence, RI, 1984.


\bibitem[K00]{K00}
Kohn, J.J.
Hypoellipticity at points of infinite type. 
Analysis, geometry, number theory: the mathematics of Leon Ehrenpreis (Philadelphia, PA, 1998), 393--398, 
Contemp. Math., 251, Amer. Math. Soc., Providence, RI, 2000. 

\bibitem[K02]{K02}
Kohn, J.J.
Superlogarithmic estimates on pseudoconvex domains and CR manifolds.
{\em Ann. of Math. (2)} {\bf 156} (2002), no. 1, 213--248. 

\bibitem[K04]{K04}
Kohn, J.J.
Ideals of multipliers. 
Complex analysis in several variables -- 
Memorial Conference of Kiyoshi Oka's Centennial Birthday, 147--157, 
Adv. Stud. Pure Math., 42, Math. Soc. Japan, Tokyo, 2004.

\bibitem[K05]{K05}
Kohn, J. J. 
Hypoellipticity and loss of derivatives. 
With an appendix by Makhlouf Derridj and David S. Tartakoff. 
{\em Ann. of Math. (2)} {\bf 162} (2005), no. 2, 943--986.

\bibitem[K10]{K10}
Kohn, J. J. 
Multipliers on pseudoconvex domains with real analytic boundaries. 
{\em Boll. Unione Mat. Ital. (9)} {\bf 3} (2010), no. 2, 309--324. 
ESI Preprint No. 2227.
{\tt http://www.esi.ac.at/static/esiprpr/esi2227.pdf}

\bibitem[K-etal04]{K-etal04}
Kohn, J. J.; Griffiths, P.A.; Goldschmidt, H.; 
Bombieri, E.; Cenkl, B.; 
Garabedian, P.; Nirenberg, L.
Donald C. Spencer (1912--2001). 
Notices Amer. Math. Soc. 51 (2004), no. 1, 17--29. 


\bibitem[KN65]{KN65}
Kohn, J. J.; Nirenberg, L.
Non-coercive boundary value problems.
{\em Comm. Pure Appl. Math.} {\bf 18} (1965), 443--492.


\bibitem[L04]{L04}
Lazarsfeld, R. 
Positivity in algebraic geometry. II. Positivity for vector bundles, and multiplier ideals. Ergebnisse der Mathematik und ihrer Grenzgebiete. 3. Folge. A Series of Modern Surveys in Mathematics [Results in Mathematics and Related Areas. 3rd Series. A Series of Modern Surveys in Mathematics], 49. Springer-Verlag, Berlin, 2004.

\bibitem[LT08]{LT08}
Lejeune-Jalabert, M.; Teissier, B.
Cl\^oture int\'egrale des id\'eaux et \'equisingularit\'e.
Annales de la facult\'e des sciences de Toulouse S\'er. 6, 
{\bf 17}, no. 4 (2008), p. 781--859.

\bibitem[M92a]{M92}
McNeal, J.D. 
Lower bounds on the Bergman metric near a point of finite type. 
{\em Ann. of Math. (2)} {\bf 136} (1992), no. 2, 339--360.

\bibitem[M92b]{M92b}
McNeal, J.D. 
Convex domains of finite type. 
{\em  J. Funct. Anal.} {\bf 108} (1992), no. 2, 361--373.
 

\bibitem[M03]{M03}
McNeal, J.D. 
Subelliptic estimates and scaling in the $\bar\d$-Neumann problem. Explorations in complex and Riemannian geometry, 197--217, 
Contemp. Math., 332, Amer. Math. Soc., Providence, RI, 2003. 

\bibitem[MN05]{MN05}
McNeal, J.D., N\'emethi, A.  
The order of contact of a holomorphic ideal in $\C^2$. {\em Math. Z.} {\bf 250} (4), 873--883.

\bibitem[MV15]{MV15}
McNeal, J.D., Varolin, D.
$L^2$ estimates for the $\bar\d$-operator. 
{\em Bull. Math. Sci.} {\bf 5} (2015), no. 2, 179--249.

\bibitem[MM17]{MM17}
McNeal, J.D., Mernik, L.
Regular versus singular order of contact on pseudoconvex hypersurfaces.
 J Geom Anal (2018) 28: 2653. 
 {\tt https://doi.org/10.1007/s12220-017-9926-9}
 
\bibitem[Mo08]{Mo08}
Mok, N.
Geometric structures on uniruled projective manifolds defined by their varieties of minimal rational tangents.
G\'eom\'etrie diff\'erentielle, physique math\'ematique, math\'ematiques
              et soci\'et\'e. II
             {\em Ast\'erisque} No. {\bf 322} (2008), 151--205.

\bibitem[Mr66]{Mr66} 
Morrey, C.~B.,~Jr. 
Multiple integrals in the calculus of variations. 
Reprint of the 1966 edition. 
Classics in Mathematics. Springer-Verlag, Berlin, 2008.

\bibitem[Na90]{N90}
Nadel, A.M.
Multiplier ideal sheaves and K\"ahler-Einstein metrics of positive scalar curvature.
{\em Ann. of Math.} (2) 132 (1990), no. 3, 549--596.

\bibitem[N12]{N12}
Nicoara, A.C.
Effective vanishing order of the Levi determinant
{\em Math. Ann.} {\bf 354} (2012), 1223--1245.

\bibitem[N14]{N14}
Nicoara, A.C.
Direct Proof of Termination of the Kohn Algorithm in the Real-Analytic Case.
Preprint 2014. 
{\tt https://arxiv.org/abs/1409.0963}
 
\bibitem[O02]{O02}
 Ohsawa, T.
 Analysis of several complex variables. Translated from the Japanese by Shu Gilbert Nakamura. Translations of Mathematical Monographs, 211. Iwanami Series in Modern Mathematics. American Mathematical Society, Providence, RI, 2002.


\bibitem[O15]{O15}
 Ohsawa, T.
 $L^2$ approaches in several complex variables. 
 Development of Oka-Cartan theory by 
 $L^2$ estimates for the $\bar\d$-operator. 
 Springer Monographs in Mathematics. Springer, Tokyo, 2015.

%


\bibitem[Si91]{Si91}
Sibony, N.
Some recent results on weakly pseudoconvex domains. 
Proceedings of the International Congress of Mathematicians, Vol. I, II (Kyoto, 1990), 943--950, Math. Soc. Japan, Tokyo, 1991. 


\bibitem[S01]{S01}
Siu, Y.-T.
Very ampleness part of Fujita's conjecture and multiplier ideal sheaves of Kohn and Nadel. Complex analysis and geometry (Columbus, OH, 1999), 171--191, Ohio State Univ. Math. Res. Inst. Publ., 9, de Gruyter, Berlin, 2001.

\bibitem[S02]{S02}
Siu, Y.-T.
Some recent transcendental techniques in algebraic and complex geometry. Proceedings of the International Congress of Mathematicians, Vol. I (Beijing, 2002), 439--448, Higher Ed. Press, Beijing, 2002. 

\bibitem[S05]{S05}
Siu, Y.-T.
Multiplier ideal sheaves in complex and algebraic geometry. Sci. China Ser. A {\bf 48} (2005), suppl., 1--31.

\bibitem[S09]{S09}
Siu, Y.-T. Dynamic multiplier ideal sheaves and the construction of rational curves in Fano manifolds. Complex analysis and digital geometry, 323--360, Acta Univ. Upsaliensis Skr. Uppsala Univ. C Organ. Hist., 86, Uppsala Universitet, Uppsala, 2009.

\bibitem[S10]{S10}
Siu, Y.-T. Effective termination of Kohn's algorithm for subelliptic multipliers.
Pure Appl. Math. Q. {\bf 6} (2010), no. 4, Special Issue: In honor of Joseph J. Kohn. Part 2, 1169--1241.

\bibitem[S17]{S17}
 Siu, Y.-T. New procedure to generate multipliers in complex Neumann problem and effective Kohn algorithm. {\em Sci. China Math.} {\bf 60} (2017), no. 6, 1101--1128.

\bibitem[St06]{St06}
Straube, E.J. 
Aspects of the $L^2$-Sobolev theory of the 
$\bar\d$-Neumann problem. International Congress of Mathematicians. Vol. II, 1453--1478, Eur. Math. Soc., ZŸrich, 2006.

\bibitem[St08]{St08}
Straube, E.J.
A sufficient condition for global regularity of the 
$\bar\d$-Neumann operator. {\em Adv. Math.} {\bf 217} (2008), no. 3, 1072--1095. 
ESI Preprint No. 
{\tt 1718: http://www.esi.ac.at/static/esiprpr/esi1718.pdf}

\bibitem[St10]{St10}
Straube, E.J. 
Lectures on the $L^2$-Sobolev theory of the $\bar\d$-Neumann problem. 
ESI Lectures in Mathematics and Physics. European Mathematical Society (EMS), ZŸrich, 2010.

\bibitem[Sw72]{Sw72}
Sweeney,~W.J. 
Coerciveness in the Neumann problem, 
{\em J. Differential Geometry} {\bf 6} (1972),
375--393.

\bibitem[Ta11]{Ta11}
Tartakoff, D.S.
Nonelliptic partial differential equations. 
Analytic hypoellipticity and the courage to localize high powers of T. Developments in Mathematics, 22. Springer, New York, 2011.

\bibitem[Tr80]{Tr80}
Tr\`eves, F.
Introduction to pseudodifferential and Fourier integral operators. Vol. 1. 
Pseudodifferential operators. The University Series in Mathematics. Plenum Press, New York-London, 1980.

\bibitem[Z17]{Z17}
Zaitsev, D.
A geometric approach to Catlin's boundary systems.
To appear in Annales de l'Institut Fourier.
{\tt https://arxiv.org/abs/1704.01808}


\bibitem[Za08]{Z08}
Zampieri, G.
Complex analysis and CR geometry. 
University Lecture Series, 43. American Mathematical Society, Providence, RI, 2008.

\end{thebibliography}
\end{document}